\documentclass[11pt]{article}

\usepackage{amsmath}
\usepackage{amsfonts}
\usepackage{amssymb}
\usepackage{amsthm}
\usepackage{enumitem}
\usepackage{mathtools}
\usepackage{setspace}
\usepackage{etoolbox}

\newtheoremstyle{mystyle}
{11pt}                          % above space 
{11pt}                          % below space 
{}                                      % body font 
{}                                      % indent amount 
{\bfseries}                     % head font 
{}                                      % post head punctuation 
{5.5pt}                         % space between head and body
{}                                      % head spec

\theoremstyle{mystyle}
\newtheorem{theorem}{Theorem}[section]

\newtheorem{lemma}[theorem]{Lemma}
\newtheorem{proposition}[theorem]{Proposition}
\newtheorem{corollary}[theorem]{Corollary}
\newtheorem{example}[theorem]{Example}

\renewenvironment{proof}[1][Proof.]{\vspace{-16.5pt} \begin{trivlist}
        \item[\hskip \labelsep {\bfseries #1}]}{\qed \end{trivlist}}

\setitemize{noitemsep, topsep=5.5pt, parsep=5.5pt, partopsep=0pt}

\allowdisplaybreaks
\appto\normalsize{
        \abovedisplayskip=5.5pt plus 2pt minus 2pt
        \belowdisplayskip=5.5pt plus 2pt minus 2pt
        \abovedisplayshortskip=5.5pt plus 2pt minus 2pt
        \belowdisplayshortskip=5.5pt plus 2pt minus 2pt}
\appto\small{
        \abovedisplayskip=5.5pt plus 2pt minus 2pt
        \belowdisplayskip=5.5pt plus 2pt minus 2pt
        \abovedisplayshortskip=5.5pt plus 2pt minus 2pt
        \belowdisplayshortskip=5.5pt plus 2pt minus 2pt}

\newcommand{\gap}{\vspace{11pt}}

\newcommand{\tr}{\operatorname{tr}}

\newcommand{\R}{\mathcal{R}}
\newcommand{\Rn}{\mathcal{R}^n}

\newcommand{\Rnp}{\mathcal{R}_+^n}
\newcommand{\Sn}{\mathcal{S}^n}

\newcommand{\Hn}{\mathcal{H}^n}
\newcommand{\V}{{\cal V}}

\newcommand{\J}{{\cal J}}
\newcommand{\wprec}{\underset{w}{\prec}}

\newcommand{\abs}[1]{\left\vert #1 \right\vert}

\newcommand{\ip}[2]{\left< #1,\, #2 \right>}

\title{\bf Some log and weak majorization inequalities  \\ 
in Euclidean Jordan algebras}
\author{
        J. Tao\\
        Department of Mathematics and Statistics\\
         Loyola University Maryland\\
         Baltimore, Maryland 21210, U.S.A\\
         jtao@loyola.edu \\ 
        *****\\
J. Jeong\\
        Applied Algebra and Optimization Research Center \\
        Sungkyunkwan University \\
        2066 Seobu-ro, Suwon 16419, Republic of Korea\\
         jjycjn@skku.edu\\
        *****\\
M. Seetharama Gowda\\
        Department of Mathematics and Statistics\\
        University of Maryland, Baltimore County\\
        Baltimore, Maryland 21250, USA\\
        gowda@umbc.edu
}

\date{\today}

\begin{document}

\maketitle

\begin{abstract}
Motivated by Horn's log-majorization (singular value) inequality $s(AB)\underset{log}{\prec} s(A)*s(B)$ and the related weak-majorization 
inequality  $s(AB)\underset{w}{\prec} s(A)*s(B)$ 
 for square complex matrices, we consider their Hermitian analogs
$\lambda(\sqrt{A}B\sqrt{A}) \underset{log}{\prec} \lambda(A)*\lambda(B)$ for positive semidefinite matrices and 
$\lambda(|A\circ B|) \underset{w}{\prec} \lambda(|A|)*\lambda(|B|)$ for general (Hermitian) matrices, where $A\circ B$ denotes the Jordan product of $A$ and $B$ and $*$ denotes the componentwise product in $\Rn$. In this paper, we extended these inequalities to the setting of Euclidean Jordan algebras in the form 
$\lambda\big (P_{\sqrt{a}}(b)\big )\underset{log}{\prec} \lambda(a)*\lambda(b)$  
 for  $a,b\geq 0$  and 
$\lambda\big (|a\circ b|\big )\underset{w}{\prec} \lambda(|a|)*\lambda(|b|)$ for all $a$ and $b$,  where $P_u$ and $\lambda(u)$ denote, respectively,
 the quadratic representation and the eigenvalue vector of an element $u$.
We also describe inequalities of the form $\lambda(|A\bullet b|)\underset{w}{\prec} \lambda({\mathrm{diag}}(A))*\lambda(|b|)$, where $A$ is a real symmetric positive semidefinite matrix  
 and $A\,\bullet\, b$ is the Schur product of $A$ and $b$. 
In the form of applications, we prove the  generalized H\"{o}lder type inequality $||a\circ b||_p\leq ||a||_r\,||b||_s$, where $||x||_p:=||\lambda(x)||_p$ denotes the spectral $p$-norm of $x$ and $p,q,r\in [1,\infty]$ with $\frac{1}{p}=\frac{1}{r}+\frac{1}{s}$. We also give precise values of the norms of the Lyapunov transformation $L_a$ and $P_a$ relative to two spectral $p$-norms. 
\end{abstract}

\vspace{1cm}
\\


%%%%%%%%%%%%%%%%%%%%%%%%%%%%%%%%%%%%%%%%%%%%%%%%%%%%
\section{Introduction}
In matrix theory, the well-known Horn's log-majorization inequality (\cite{hiai-petz}, Corollary 6.14) asserts that for any two $n\times n$ complex matrices $A$ and $B$, 
\begin{equation}\label{horn-log}
s(AB)\underset{log}{\prec} s(A)*s(B),
\end{equation}
 where $s(X)$ denotes the vector of singular values of $X$ written in the decreasing order  and 
$*$ denotes the componentwise product of vectors in $\Rn$. A simple consequence of (\ref{horn-log}) is the weak-majorization inequality
\begin{equation}\label{horn-weak}
s(AB)\underset{w}{\prec} s(A)*s(B).
\end{equation}

When $A$ and $B$ are Hermitian, the product $AB$ need not be Hermitian and so, 
to stay within the realm of Hermitian matrices, we consider the following Hermitian analogs. 

\begin{itemize}
\item [$\bullet$] When $A$ and $B$ are (Hermitian and) positive semidefinite,
$$\lambda\Big (\sqrt{A}B\sqrt{A}\Big ) \underset{log}{\prec} \lambda(A)*\lambda(B),$$
where, for a complex Hermitian matrix $X$, $\lambda(X)$ denotes the vector of eigenvalues of $X$ written in the decreasing order.
This can be seen from (\ref{horn-log}) by using the fact that the eigenvalues of $AB$ are the same as those of $\sqrt{A}B\sqrt{A}$, 
see the proof of Corollary III.4.6 in \cite{bhatia}. \\
As an immediate consequence, we have
$$\lambda\Big (\sqrt{A}B\sqrt{A}\Big ) \underset{w}{\prec} \lambda(A)*\lambda(B).$$
\item [$\bullet$] When $A$ and $B$ are Hermitian, 
$$\lambda(|A\circ B|) \underset{w}{\prec} \lambda(|A|)*\lambda(|B|),$$
where $A\circ B:=\frac{AB+BA}{2}$ is the Jordan product of $A$ and $B$, and 
$|A|:=\sqrt{A^*A}$, etc. This inequality can be seen from (\ref{horn-log}) by using the majorization inequality $s(X+Y)\underset{w}{\prec} s(X)+s(Y)$ (\cite{hiai-petz}, Corollary 6.12), 
see Section 4 for an elaboration.
\end{itemize}

Now, these Hermitian inequalities can be viewed as describing (pointwise) eigenvalue majorization inequalities  of two linear transformations $L_A$ and $P_A$ on 
the space $\Hn$ of all $n\times n$ complex Hermitian matrices, where
$$L_A(X):=A\circ X\quad\mbox{and}\quad P_A(X):=AXA\quad (X\in \Hn).$$
Noting that $\Hn$ is a Euclidean Jordan algebra with the Jordan product $X\circ Y: =\frac{XY+YX}{2}$ and the inner product $\langle X,Y\rangle:=tr(XY)$, we  ask if such inequalities can be established in the setting of general Euclidean Jordan algebras. The main objective of this paper is to 
provide an affirmative answer.  
\\

Let $(\V, \circ, \langle\cdot,\cdot\rangle)$ denote a Euclidean Jordan algebra of rank $n$ with $\circ$ denoting the Jordan product and $\langle \cdot,\cdot\rangle$ denoting the (trace) inner product; $\Sn$, the space of all $n\times n$ real symmetric matrices, and $\Hn$, the space of all $n\times n$ complex Hermitian matrices are two primary examples. For $x\in \V$, let $\lambda(x)$ denote the vector of eigenvalues of 
$x$ written in the decreasing order. Given $x,y\in \V$, we say that $x$ is {\it weakly majorized} by $y$ and write $x\underset{w}{\prec} y$ if 
$\lambda(x) \underset{w}{\prec} \lambda(y)$ in $\Rn$, which means that for each index $k$, $1\leq k\leq n$,
$\sum_{i=1}^{k}\lambda_i(x)\leq \sum_{i=1}^{k}\lambda_i(y);$
additionally, if the equality holds when $k=n$, we say that $x$ is {\it majorized} by $y$ and write $x\prec y$. 
Replacing sums by products, one defines {\it weak log-majorization} and {\it log-majorization}, respectively denoted by  $x\underset{wlog}{\prec} y$ and 
$x\underset{log}{\prec} y$.
The above concepts have been extensively studied in matrix theory and other areas, see for example, \cite{marshall-olkin, bhatia, hiai-petz}. 
In the setting of Euclidean Jordan algebras, they have been recently studied in several papers  
\cite{tao et al, wang et al, gowda-doubly stochastic, gowda-holder, gowda-majorization}.
\\

On $\V$ we consider two linear transformations,  $L_a$, the {\it Lyapunov transformation of $a$}, and $P_a$, 
the {\it quadratic representation of $a$}, defined on $\V$ as follows: 
\begin{equation}\label{la and pa}
L_a(x):=a\circ x\quad \mbox{and}\quad P_a(x):=2\,a\circ (a\circ x)-a^2\circ x\quad (a, x\in \V).
\end{equation}
(It turns out that in the algebras $\Sn$ and $\Hn$,   $P_A(X)=AXA.$)
\\

Generalizing the (Hermitian) matrix inequalities stated above, we show the following:  For $a,b\in \V$, 
\begin{equation} \label{nonnegative majorization inequality}
\lambda\big (P_{\sqrt{a}}(b)\big )\underset{log}{\prec} \lambda(a)*\lambda(b)\quad (a,b\geq 0),
\end{equation}

\begin{equation} \label{Pa absolute majorization inequality}
\lambda\big (|P_{a}(b)|\big )\underset{w}{\prec} \lambda(a^2)*\lambda(|b|),
\end{equation}
and
\begin{equation}\label{La absolute majorization inequality}
\lambda\big (|L_{a}(b)|\big )=\lambda\big (|a\circ b|\big )\underset{w}{\prec} \lambda(|a|)*\lambda(|b|).
\end{equation}

One novelty here is that our proofs do not rely on matrix theory results. Rather, they are based on general results on Euclidean Jordan algebras
and on well known majorization inequalities in $\Rn$.
The proof of (\ref{nonnegative majorization inequality}) is modeled after a proof  of Gelfand-Naimark Theorem presented in \cite{hiai}  (see \cite{hiai-petz}, Chapter 6,
 where it is attributed to \cite{li-mathias}).
We prove (\ref{Pa absolute majorization inequality})  by combining (\ref{nonnegative majorization inequality}) with the inequality 
\begin{equation} \label{positive majorization inequality}
|P(b)|\underset{w}{\prec}P(|b|)
\end{equation}
that is valid for all positive linear transformations $P$ on $\V$ and $b$.
The inequality (\ref{La absolute majorization inequality}) is established 
by relying heavily on the so-called Fan-Theobald-von Neumann inequality \cite{baes}
$$\langle x,y\rangle\leq \langle \lambda(x),\lambda(y)\rangle\quad (x,y\in \V).$$

As we shall elaborate in the subsequent sections, both $P_a(x)$ and $L_a(x)$ can be described by means of Schur products of the form 
$A\bullet x$ (relative to the Peirce decomposition coming from a Jordan frame), where $A=[a_ia_j]$ in the case of 
$P_a$ and $A=[\frac{a_i+a_j}{2}]$ in the case of $L_a$, with $a_1,a_2,\ldots, a_n$ denoting  the eigenvalues of $a$. Then, 
the inequalities (\ref{Pa absolute majorization inequality}) and (\ref{La absolute majorization inequality}) take the form
\begin{equation}\label{schur product absolute value inequality}
 \lambda(|A\bullet b| ) \underset{w}{\prec} \lambda (|{\mathrm{diag}}(A)| )*\lambda(|b|),
\end{equation}
where $\rm{diag}(A)$ denotes the diagonal vector of $A$. 

Going beyond the case of $A=[a_ia_j]$, we show that (\ref{schur product absolute value inequality}) is valid for all real symmetric positive semidefinite matrices $A$ and 
raise the {\it problem of characterizing $A$ for which (\ref{schur product absolute value inequality}) holds.}
 
In the form of applications, we describe some norm inequalities.  Given $p\in [1,\infty]$, we consider the spectral $p$-norm on $\V$: $||x||_p:=||\lambda(x)||_p$, 
where  the latter norm is computed in $\Rn$. Under the assumption that (\ref{schur product absolute value inequality}) holds for all $b$, we show that 
$$||A\bullet b||_p\leq ||\mathrm{diag}(A)||_r\,||b||_s \quad (b\in \V)$$
whenever $p,q,r\in [1,\infty]$ with $\frac{1}{p}=\frac{1}{r}+\frac{1}{s}$. When specialized, this yields 
the  generalized  H\"{o}lder type inequality 
$$||a\circ b||_p\leq ||a||_r\,||b||_s \quad (a,b\in \V).$$ 
We also compute the norms of $L_a$ and $P_a$ relative to two spectral $p$-norms. 
\\

The organization of the paper is as follows. In Section 2, we cover some preliminary material on Euclidean Jordan algebras and majorization results in $\R^n$.
In Section 3, we will provide the proof of the log-majorization inequality (\ref{nonnegative majorization inequality}), describe results on positive transformations, and 
prove the inequality (\ref{Pa absolute majorization inequality}). Here, we will also  prove the inequality (\ref{schur product absolute value inequality}) for real symmetric positive semidefinite matrices.
Section 4 deals with the proof of the inequality (\ref{La absolute majorization inequality}). Applications, describing the generalized H\"{o}lder type inequality and norms of $L_a$ and $P_a$ will be covered in Section 5.  

%%%%%%%%%%%%%%%%%%%%%%%%%%%%%%%%%%%%%%%%%%%%%%%%%%%%%%%%%%%%%%%%%%%%%%%%%%%%%%%%%%%%%%%%%%%%%%%%

\section{Preliminaries}
Throughout this paper, $\Rn$ denotes the Euclidean $n$-space whose elements are regarded as column vectors or row vectors depending on the context. For elements $p,q\in \Rn$, $p*q$ denotes their componentwise product. For $p\in \Rn$, $p^\downarrow$ and $|p|$ denote, respectively,  the  decreasing rearrangement  and the vector of absolute values of entries of $p$. Borrowing the notation used in the Introduction, we recall  the following results.

\begin{proposition}\label{prop1} (\cite{bhatia}, Problem II.5.16, Example II.3.5, Exercise II.3.2;  \cite{marshall-olkin}, A.7.(ii), p. 173 and  \cite{cvetkovski}, p.136) 
{\it \begin{itemize}
\item [$(a)$] Let $p, \, q \in \Rn$, $r\in \Rnp$, and $p \underset{w}{\prec}  q$. Then
$p^{\downarrow} *r^{\downarrow} \underset{w}{\prec} q^{\downarrow}*r^{\downarrow}.$
\item [$(b)$] Let $p, \, q \in \Rn$. Then $p\prec q \Rightarrow |p| \underset{w}{\prec} |q|$.
\item [$(c)$] Let $p, \, q \in \Rn$ and $I$ be an interval in $\R$ that contains all the entries of $p$ and $q$. 
Then $p \wprec  q \,\,\Longleftrightarrow \sum_{i=1}^{n}\phi(p_i)\leq \sum_{i=1}^{n}\phi(q_i)$
for every increasing convex function $\phi: I\rightarrow \R$. Moreover, if $p \wprec  q $ and $\sum_{i=1}^{n}\phi(p_i)= \sum_{i=1}^{n}\phi(q_i)$ for some increasing strictly convex 
function $\phi$, then 
$p^{\downarrow}=q^{\downarrow}$.
\end{itemize}
}
\end{proposition}

\gap

Throughout, we let  $(\V, \circ,\langle\cdot,\cdot\rangle)$ denote a Euclidean Jordan algebra of rank 
$n$ with unit element $e$ \cite{faraut-koranyi, gowda-sznajder-tao};
the Jordan product and inner product of  elements $x$ and $y$ in $\V$ are respectively denoted by   $x\circ y$  and $\langle x,y\rangle$. 
It is well known \cite{faraut-koranyi} that any Euclidean Jordan algebra is a direct product/sum 
of simple Euclidean Jordan algebras and every simple Euclidean Jordan algebra is isomorphic to one of five algebras, 
three of which are the algebras of $n\times n$ real/complex/quaternion Hermitian matrices. The other two are: the algebra of $3\times 3$ octonion Hermitian matrices and the Jordan spin algebra.
In the algebras $\Sn$ (of all $n\times n$ real symmetric matrices) and $\Hn$ (of all $n\times n$ complex Hermitian matrices), the Jordan product and the inner product are given, respectively, by 
$$X\circ Y:=\frac{XY+YX}{2}\quad\mbox{and}\quad \langle X,Y\rangle:=\tr(XY),$$
where the trace of a real/complex matrix is the sum of its diagonal entries. 

According to the {\it spectral decomposition 
theorem} \cite{faraut-koranyi}, any element $x\in \V$ has a decomposition
$$x=x_1e_1+x_2e_2+\cdots+x_ne_n,$$
where the real numbers $x_1,x_2,\ldots, x_n$ are (called) the eigenvalues of $x$ and 
$\{e_1,e_2,\ldots, e_n\}$ is a Jordan frame in $\V$. (An element may have decompositions coming from different Jordan frames, but the eigenvalues remain the same.) Then, $\lambda(x)$ -- called the {\it eigenvalue vector} of $x$ -- is the vector of eigenvalues of $x$ written in the decreasing order. We write
$$\lambda(x)=\big ( \lambda_1(x),\lambda_2(x),\ldots, \lambda_n(x)\big ).$$ 
It is known that $\lambda:\V\rightarrow \Rn$ is  continuous \cite{baes}. 
\\

We recall some standard definitions/results. The {\it rank} of an element $x$ is the number of nonzero eigenvalues. An element  $x$ is said to be {\it invertible} if all its eigenvalues are nonzero; such elements form a dense subset of $\V$.
We use the notation $x\geq 0$ ($x> 0$) when all the eigenvalues of $x$ are nonnegative (respectively, positive) and $x\geq y$ when $x-y\geq 0$, etc. The set of all elements $x\geq 0$ (called the symmetric cone of $\V$) is a self-dual cone. 
Given the spectral decomposition $x=x_1e_1+x_2e_2+\cdots+x_ne_n$, we define $|x|:=|x_1|e_1+|x_2|e_2+\cdots+|x_n|e_n$ and 
$\sqrt{x}:=\sqrt{x_1}e_1+\sqrt{x_2}e_2+\cdots+\sqrt{x_n}e_n$ when $x\geq 0$.
The {\it trace and determinant} of $x$ are defined by
$\tr(x):=x_1+x_2+\cdots+x_n\, \mbox{and} \, \det(x):=x_1x_2\cdots x_n.$
Also, for $p\in [1,\infty]$,  
we define the corresponding {\it spectral norm} $||x||_p := (\sum_{i=1}^{n} |x_i|^p)^{1/p}$ when $p<\infty$ and 
$||x||_{\infty}= \mbox{max}_i |x_i|$.
\\

For any $a\in \V$, we define $L_a$ and $P_a$ as in (\ref{la and pa}). We say that two elements $a$ and $b$ {\it operator commute} if $L_a$ and $L_b$ commute, or equivalently, if $a$ and $b$ have their spectral representations with respect to the same Jordan frame.
\\ 

An element $c\in \V$ is an {\it idempotent} if $c^2=c$; it is said to be a {\it primitive idempotent} if it is nonzero and cannot be written as the sum of two other nonzero idempotents. We write
$\J(\V)$ for the set of all primitive idempotents and $\J^{(k)}(\V)$ ($1\leq k\leq n$) for the set of all idempotents of rank $k$.

It is  known that $(x,y)\mapsto \tr(x\circ y)$ defines another inner product on $\V$ that is compatible with the Jordan product. 
{\it Throughout this paper, we assume that the inner product on $\V$ 
is this trace inner product, that is,
$\langle x,y\rangle=\tr(x\circ y).$}
In this inner product, the norm of any primitive element is one and so any Jordan frame
in $\V$ is an orthonormal set. Additionally, $\tr(x)=\langle x,e\rangle \,\,\mbox{for all}\,\,x\in \V.$
\\

Given an idempotent $c$, we have the  {\it Peirce (orthogonal) decomposition} \cite{faraut-koranyi}:  $\V=\V(c,1)+\V(c,\frac{1}{2})+\V(c,0)$, where $\V(c,\gamma)=\{x\in \V: x\circ c=\gamma\,x\}$ with $\gamma\in \{1,\frac{1}{2},0\}$. Then, any $b\in \V$ can be decomposed as $b=u+v+w$, where  $u\in \V(c,1)$, $v\in \V(c,\frac{1}{2})$, and $w\in \V(c,0)$. 

\gap
 
Below, we record  some standard results that are needed in the sequel. We emphasize that $\V$ has rank $n$ and carries the trace inner product.

\begin{proposition}\label{prop2} (\cite{tao et al}, Theorem 6.1 or \cite{gowda-doubly stochastic}, page 54)
{\it If $c$ is an idempotent and  $x=u+v+w$, where $u\in \V(c, 1)$, $v\in \V(c, \frac{1}{2})$ and $w\in \V(c, 0)$, then
$\lambda (u+w) \prec \lambda (x).$
}\end{proposition}

\begin{proposition}\label{prop3}  
{\it 
Let  $x,y,c,u \in \V$. Then, 
\begin{itemize}
\item [$(i)$] $S_k(x):=\lambda_1(x)+ \lambda_2(x) + \cdots + \lambda_k(x) = \underset{c\in \J^{k}(\V)}{\max} \langle x, c\rangle$.
\item [$(ii)$] $\langle x,y\rangle \leq \langle \lambda(x),\lambda(y)\rangle\leq \langle \lambda(|x|),\lambda(|y|)\rangle.$ 
\item [$(iii)$] $x\leq y\Rightarrow \lambda(x)\leq \lambda(y)$.
\item [$(iv)$] $P_c$ is a positive transformation, that is,  $x\geq 0\Rightarrow P_c(x)\geq 0$.
\item [$(v)$] $\det P_c(u)=(\det c)^2\,\det u.$
\item [$(vi)$] $||x\circ y||_p \leq ||x||_p \, ||y||_{\infty}$, where $p\in [1,\infty]$. In particular,  $||x\circ y||_\infty \leq ||x||_\infty \, ||y||_{\infty}$, or\\ equivalently, 
$\lambda_1(|x\circ y|)\leq \lambda_1(|x|)\,\,\lambda_1(|y|)$.
\end{itemize}
}
\end{proposition}

In the above proposition, Item $(i)$ is stated in \cite{baes}, Lemma 20. The first inequality in $(ii)$, known as  the Fan-Theobald-von Neumann inequality, can be found in 
\cite{baes}, Theorem 23; the second inequality -- a particular case of the first inequality -- is a simple consequence of the  rearrangement inequality of Hardy-Littlewood-P\'{o}lya in $\Rn$: 
$\langle p,q\rangle \leq \langle p^\downarrow,q^\downarrow \rangle$. 
Item $(iii)$ is a consequence of the the well-known min-max theorem of Hirzebruch \cite{hirzebruch, gowda-tao-cauchy}.
 Items $(iv)$ and $(v)$ are well-known properties of $P_c$, see, e.g., \cite{faraut-koranyi}. Items in $(vi)$ follow from Theorem 3.2 in \cite{gowda-holder}.

\gap

Given a Jordan frame $\{e_1,e_2,\ldots, e_n\}$ in $\V$, we consider the corresponding Peirce decomposition of $\V$ and $x\in \V$ (\cite{faraut-koranyi}, Theorem IV.2.1):
$\V=\sum_{i\leq j}\V_{ij}$ and  $x=\sum_{i\leq j}x_{ij}$, where $x_{ij}\in \V_{ij}$ for all $1\leq i\leq j\leq n$. Here, $\V_{ii}:=\R\,e_i$ for all $i$ and 
and $\V_{ij}:=\V(e_i,\frac{1}{2})\cap \V(e_j,\frac{1}{2})$ for $i<j$. Then, for any matrix $A=[a_{ij}]\in \Sn$, we define
\begin{equation} \label{bullet product}
A\bullet x:=\sum_{i\leq j}a_{ij}\,x_{ij}.
\end{equation}
We call this the {\it Schur product of $A$ and $x$} relative to the Jordan frame $\{e_1,e_2,\ldots, e_n\}$.
(This is a generalization of the usual Schur/Hadamard product of matrices in $\Sn$.) \\
Two primary examples: For any $a\in \V$ with spectral decomposition $a=\sum_{i=1}^{n}a_ie_i$,  
\begin{center}
$L_a(x)=A\bullet x$, where $A=[\frac{a_i+a_j}{2}]$ and  $P_a(x)=A\bullet x$, where $A=[a_ia_j]$.
\end{center}
We refer to  \cite{gowda-tao-sznajder} for further examples and properties.
We record a recent result that connects Schur products  and  quadratic representations.

\gap

\begin{proposition}\label{prop4} (\cite{gowda-majorization}, Corollary 3.4)
{\it 
Consider a Jordan frame $\{e_1,e_2,\ldots, e_n\}$. Suppose $A=[a_{ij}]\in \Sn$ is positive semidefinite and  let $a=a_{11}e_1+\cdots+a_{nn}e_n$.
Then, $A\bullet x\prec P_{\sqrt{a}}(x)$ for all $x\in \V$.
}
\end{proposition}

%%%%%%%%%%%%%%%%%%%%%%%%%%%%%%%%%%%%%%%%%%%%%%%%%%%%%%%%%%%%%%%%%%%%%%%%%%%%%%%%%%%%%%%%%%%%%
\section{A log-majorization inequality}

In this section, we prove the inequality (\ref{nonnegative majorization inequality}). 
Its proof (along with that of Lemma \ref{lemma 3} below) is  
modeled after techniques given in \cite{hiai} (see also, \cite{hiai-petz}, Theorem 6.13 and \cite{li-mathias}).
First, we present several lemmas.

\gap
 
\begin{lemma}\label{lemma 1} (\cite{lim}, Corollary 9)
{\it For $a,b\geq 0$ in $\V$, $\lambda\big (P_{\sqrt{a}}(b)\big )=\lambda\big (P_{\sqrt{b}}(a)\big ).$ 
}
\end{lemma}

\begin{lemma}\label{lemma 2}
{\it For $a,b\geq 0$ in $\V$, $\lambda\big (P_{\sqrt{a}}(b)\big )\leq ||a||_\infty\,
\lambda(b).$
}
\end{lemma}

\begin{proof} As $a\geq 0$, we have $\lambda_{1}(a)=||a||_\infty$ and so,
$a\leq \lambda_{1}(a)\,e=||a||_\infty\,e.$ Since for any $c$, $P_c$ is a positive (linear) transformation,  $P_{\sqrt{b}}(a)\leq ||a||_\infty\,P_{\sqrt{b}}(e)=||a||_\infty\,b.$ 
From Lemma \ref{lemma 1} and Item $(iii)$ in Proposition \ref{prop3}, 
$$\lambda\big (P_{\sqrt{a}}(b)\big )=\lambda\big (P_{\sqrt{b}}(a)\big )\leq 
\lambda\big (||a||_\infty\,b\big )=||a||_\infty\,\lambda(b).$$
\end{proof}

\begin{lemma}\label{lemma 3} 
{\it Suppose $a\in \V$ is invertible with its spectral decomposition
$$a=\sum_{i=1}^{n}a_ie_i=\sum_{i=1}^{n}|a_i|\varepsilon_i\,e_i,$$
where $|a_1|\geq |a_2|\geq \cdots\geq |a_n|$ and $\varepsilon_i=sgn(a_i)$ for all $i$. For $1\leq k\leq n$, let 
$$x:=\sum_{i=1}^{k}\frac{|a_i|}{|a_k|}\,e_i+\sum_{j=k+1}^{n}e_j\quad\mbox{and}\quad y:=\sum_{i=1}^{k}|a_k|\varepsilon_ie_i+\sum_{j=k+1}^{n}a_je_j.$$ Then, the following statements hold:
\begin{itemize}
\item [$(i)$] $x\geq e$,
\item [$(ii)$] $x$ and $y$ operator commute,
\item [$(iii)$] $P_{\sqrt{x}}(y)=a$, $P_{x}(y^2)=a^2$,
\item [$(iv)$] $(\det x)\,||y||_\infty^k=\prod_{i=1}^{k} |a_i|.$
\end{itemize}
}
\end{lemma}

This lemma follows from direct verification.

\begin{theorem}\label{log-majorization theorem}
{\it Let $a,b\geq 0$ in $\V$. Then,
$$\lambda\big (P_{\sqrt{a}}(b)\big )\underset{log}{\prec} \lambda(a)*\lambda(b).$$  
Consequently, $\lambda\big ( P_{\sqrt{a}}(b)\big )\underset{w}{\prec}\lambda(a)*\lambda(b).$
}
\end{theorem}

\gap

\begin{proof} 
We  have to show that for all  $k\in \{1,2,\ldots, n\}$,
\begin{equation}\label{inequality for all k}
\prod_{i=1}^{k}\lambda_i\big (P_{\sqrt{a}}(b)\big )\leq \prod_{i=1}^{k}\lambda_i(a)\,\lambda_i(b)
\end{equation}
 with equality when $k=n$. By continuity, it is enough to prove this statement for $a,b>0$. 
So, in the rest of the proof, {\it we assume that $a,b>0$ and fix  $k$.}\\
Corresponding to the spectral decomposition $a=\sum_{i=1}^{n}a_ie_i$ (with $a_1\geq a_2\geq \cdots\geq a_n>0$) and $k$, we define $x$ and $y$ as in Lemma \ref{lemma 3}:
$$x:=\sum_{i=1}^{k}\frac{a_i}{a_k}\,e_i+\sum_{j=k+1}^{n}e_j\quad\mbox{and}\quad y:=\sum_{i=1}^{k}a_ke_i+\sum_{j=k+1}^{n}a_je_j.$$ 
In addition to the properties of $x$ and $y$ listed in Lemma \ref{lemma 3}, we observe that $x,y\geq 0$, $\sqrt{y}\,\circ \sqrt{x}=\sqrt{a}$, and $P_{\sqrt{y}}P_{\sqrt{x}}=P_{\sqrt{a}}.$
Now, $x\geq e$ implies, via Lemma \ref{lemma 1} and Proposition \ref{prop3}, 
$$\lambda_i\big (P_{\sqrt{x}}(b)\big )=\lambda_i\big (P_{\sqrt{b}}(x)\big )\geq \lambda_i\big (P_{\sqrt{b}}(e)\big )=\lambda_i(b) 
\,\,\mbox{for all}\,\,i=1,2,\ldots, n.$$ 
Hence, 
$$\frac{\lambda_i\big (P_{\sqrt{x}}(b)\big )}{\lambda_i(b)}\geq 1\,\,\mbox{for all}\,\,i=1,2,\ldots,n.$$
It follows that 
$$ \prod_{i=1}^{k}\frac{\lambda_i\big (P_{\sqrt{x}}(b)\big )}{\lambda_i(b)}\leq  \prod_{i=1}^{n}\frac{\lambda_i(P_{\sqrt{x}}(b)\big )}{\lambda_i(b)}=\frac{\det P_{\sqrt{x}}(b)}{\det b}=\det x.$$
This implies that 
$$\prod_{i=1}^{k}\lambda_i\big (P_{\sqrt{x}}(b)\big )\leq (\det x)\,\prod_{i=1}^{k}\lambda_i(b).$$
On the other hand, for any index $i\in \{1,2,\ldots, n\}$, from Lemma \ref{lemma 2},
$$\lambda_i\big (P_{\sqrt{a}}(b)\big )=\lambda_i\big (P_{\sqrt{y}}P_{\sqrt{x}}(b)\big )\leq ||y||_\infty\,\lambda_i\big (P_{\sqrt{x}}(b)\big ).$$
Hence,
$$\prod_{i=1}^{k}\lambda_i\big (P_{\sqrt{a}}(b)\big )\leq ||y||_\infty^k\,\prod_{i=1}^{k}\lambda_i\big (P_{\sqrt{x}}(b)\big )\leq ||y||_\infty^k\, (\det x)\,\prod_{i=1}^{k}\lambda_i(b).$$
As $(\det x)||y||_\infty^k=\prod_{i=1}^{k}a_i=\prod_{i=1}^{k}\lambda_i(a)$ (from Lemma \ref{lemma 3}), we see that 
$$\prod_{i=1}^{k}\lambda_i\big (P_{\sqrt{a}}(b)\big )\leq \prod_{i=1}^{k} \lambda_i(a)\,\lambda_i(b).$$
This proves the inequality (\ref{inequality for all k}). Now suppose $k=n$.
Then,
$$\prod_{i=1}^{n}\lambda_i\big (P_{\sqrt{a}}(b)\big )=\det \,P_{\sqrt{a}}(b)=(\det a)\det b= \prod_{i=1}^{n} \lambda_i(a)\,\lambda_i(b).$$
Finally, the weak-majorization inequality is an immediate consequence of the log-majorization inequality. This completes the proof.
\end{proof}

\gap

As a consequence of the above log-majorization inequality, we now prove a weak-majorization inequality dealing with quadratic representations.
While our primary focus here is to prove  (\ref{Pa absolute majorization inequality}), it turns out that a general result dealing with L{\"o}wner maps of sublinear functions 
can obtained without much difficulty. First, some relevant definitions.
\\
For  a function $\phi : \R \to \R$, the corresponding {\it L{\"o}wner map} (also denoted by) $\phi : \V \to \V$ is defined as follows: For
any $x\in \V$ with spectral decomposition
$x = \sum_{i=1}^{n} x_i e_i$, $\phi(x):=\sum_{i=1}^{n} \phi(x_i) e_i.$
We make one simple observation (using \cite{bhatia}, Corollary II.3.4): When $\phi$ is convex,
$$x\prec y\Rightarrow \phi(x)\wprec \phi(y).$$

A function $\phi : \R \to \R$ is said to be {\it sublinear} if
\begin{enumerate}
\item $\phi(\mu t) = \mu\phi(t)$ for all $\mu \geq 0$ and $t \in \R$;
\item  $\phi(t+s) \leq \phi(t) + \phi(s)$ for all $t,s\in \R$.
\end{enumerate}

It is easy to see that sublinear functions on $\R$ are of the form $\phi(t)=\alpha\,t$ for $t\geq 0$ and $\phi(t)=\beta\,t$ for $t\leq 0$, where (constants) $\alpha,\beta\in \R$ satisfy 
$\beta\leq \alpha$. Particular examples are $\phi(t)=|t|, \max\{t,0\}$, and $\max\{-t,0\}$.

\gap

\begin{theorem} \label{Pa weak-majorization}
{\it Let $\phi : \R \to \R$ be a nonnegative sublinear function. Then, for all $a,b\in \V$,
$$\lambda\big ( \phi(P_a(b))\big )\underset{w}{\prec}\lambda(a^2)*\lambda(\phi(b)).$$
In particular, 
 $\lambda\big ( |P_a(b)|\big )\underset{w}{\prec}\lambda(a^2)*\lambda(|b|).$
}
\end{theorem}

We prove the above result by relying on the following lemma that may be of independent interest. It is motivated by a recent result in \cite{jeong et al} which states that 
$\phi(P_u(x))\underset{w}{\prec} P_u(\phi(x))$ when $\phi$ is a convex function  on $\R$ with $\phi(0)\leq 0$ and $u^2\leq e$.
In the result below, we replace $P_u$ by a positive linear transformation, but restrict $\phi$ to sublinear functions.

\begin{lemma}\label{positive transformation lemma}
{\it 
        Suppose $P : \V \to \V$ is a positive linear transformation and $\phi : \R \to \R$ is sublinear. Then,
        \[ \phi(P(x)) \wprec P(\phi(x)) \quad (x \in \V). \]
}
\end{lemma}

\begin{proof}
        We first claim, for any $x \in \V$ and $a \geq 0$, the inequality
        \[ \phi(\ip{x}{a}) \leq \ip{\phi(x)}{a}. \]
         To see this, we write the spectral decomposition of $x$ and use the definition of $\phi(x)$:  
        \[ x = \sum_{i=1}^{n} x_i e_i \quad \mbox{and}\quad \phi(x) = \sum_{i=1}^{n} \phi(x_i) e_i. \]
        Since $\ip{e_i}{a} \geq 0$ for all $i$ and $\phi$ is sublinear, it follows that
        \[ \phi(\ip{x}{a}) = \phi \left( \sum_{i=1}^{n} x_i \ip{e_i}{a} \right) \leq \sum_{i=1}^{n} \phi(x_i) \ip{e_i}{a} = \ip{\phi(x)}{a}. \]
        Hence the claim.

        \gap

        Now,  we write the spectral decomposition of $P(x)$  and use the definition of $\phi(P(x))$:
        \[ P(x) = \sum_{i=1}^{n} \alpha_i f_i, \quad \phi(P(x)) = \sum_{i=1}^{n} \phi(\alpha_i) f_i. \]
        Since the eigenvalues of $\phi(P(x))$ are $\phi(\alpha_1),\, \ldots,\, \phi(\alpha_n)$, there exists a permutation $\sigma$ on the set   $\{1,\, 2,\, \ldots,\, n\}$
such that $\lambda_i(\phi(P(x))) = \phi(\alpha_{\sigma(i)}) = \phi(\ip{P(x)}{f_{\sigma(i)}})$ for all $i$. To simplify the notation, we let $i' := \sigma(i)$.
 Then, for any index $k$, $1 \leq k \leq n$, we have
        \[ S_k(\phi(P(x))) = \sum_{i=1}^{k} \lambda_i(\phi(P(x))) = \sum_{i=1}^{k} \phi(\ip{P(x)}{f_{i'}}) = \sum_{i=1}^{k} \phi(\ip{x}{P^{\ast}(f_{i'})}), \]
        where $P^{\ast}$ is the adjoint of $P$. As the symmetric cone of $\V$  is  self-dual, we see that $P^*$ is also a positive linear transformation, hence 
 $P^{\ast}(f_{i'}) \geq 0$. By applying the above claim, we  see that
        \[ \phi(\ip{x}{P^{\ast}(f_{i'})}) \leq \ip{\phi(x)}{P^{\ast}(f_{i'})} = \ip{P(\phi(x))}{f_{i'}}. \]
        Hence,
        \[ S_k(\phi(P(x))) \leq \sum_{i=1}^{k} \ip{P(\phi(x))}{f_{i'}} \leq \max_{c \in \mathcal{J}^{(k)}(\V)} \ip{P(\phi(x))}{c} = S_k(P(\phi(x))), \]
        where the last equality follows from Item $(i)$ in Proposition 2.3. As this inequality holds for all $1 \leq k \leq n$, we have $\phi(P(x)) \wprec P(\phi(x))$.
\end{proof}

\begin{example}
        Let $P$ be as in the above lemma. Taking $\phi(t) = \abs{t}$, $\phi(t) = \max\{t,\, 0\}$, or $\phi(t) = \max\{-t,\, 0\}$, we get the inequalities
        \[ \abs{P(x)} \wprec P(\abs{x}), \quad P(x)^{+} \wprec P(x^{+}), \quad \mbox{and}\,\,P(x)^{-} \wprec P(x^{-}), \]
        for any $x \in \V$.
\end{example}

\gap

\noindent{\bf Proof of Theorem \ref{Pa weak-majorization}}: 
For the given $a,b\in \V$, we have $a^2,|b|\geq 0$ and $\sqrt{a^2}=|a|$. By Theorem \ref{log-majorization theorem}, $\lambda\big (P_{\sqrt{a^2}}(|b|)\big )\underset{log}{\prec} \lambda(a^2)*\lambda(|b|)$. 
Since the inequality $p\underset{log}{\prec} q$ in $\Rnp$ implies that $p\underset{w}{\prec} q$ (see \cite{bhatia}, Example II.3.5), we see that 
$$\lambda\big (P_{|a|}(|b|)\big )\underset{w}{\prec} \lambda(a^2)*\lambda(|b|).$$
Now, it is known that 
$P_a(x)\prec P_{|a|}(x)$ for all $x$, see e.g., \cite{gowda-majorization}, page 11. Using this and the above lemma with $P_a$ in place of $P$, we have
$$\lambda\big (\phi(P_a(b))\big )\underset{w}{\prec} \lambda\big (P_a(\phi(b))\big )
\prec\lambda\big (|P_{|a|}(\phi(b))\big )\underset{w}{\prec} 
\lambda(a^2)*\lambda(\phi(b)),$$
where we have used the condition that $\phi$ is nonnegative (so $\phi(t)\geq 0$ for all $t\in \R$, consequently,  $\phi(b)\geq 0$.)  Finally, by putting $\phi(t)=|t|$, we get  
$ \lambda(|P_a(b)|)\underset{w}{\prec} \lambda(a^2)*\lambda(|b|).$\hfill $\Box$

\gap

\noindent{\bf Remark.} In a recent paper \cite{gowda-holder},  it was shown that
$P_a(x)\prec L_{a^2}(x)\, \mbox{for all}\,\,a, x\in \V$. This, in particular, implies that
\begin{equation}\label{gowda-majorization}
\lambda\big (P_{\sqrt{a}}(b)\big )\prec \lambda(a\circ b)\quad \mbox{for all}\,\,a\geq 0\,\,\mbox{and}\,\,b\in \V.
\end{equation}
In the next section, we will prove the inequality $\lambda(|a\circ b|)\underset{w}{\prec}\lambda(|a|)*\lambda(|b|)$. 
Based on these two results, we can give an alternative proof of the absolute-value case in the above theorem as follows:
$$P_a(b)\prec a^2\circ b\Rightarrow  \lambda(|P_a(b)|)\underset{w}{\prec} \lambda(|a^2\circ b|)\underset{w}{\prec} \lambda(a^2)*\lambda(|b|).$$
 
\gap

We now extend Theorem \ref{Pa weak-majorization} to Schur products. In what follows, for a matrix $A\in \Sn$, ${\mathrm{diag}}(A)$ denotes the diagonal vector of $A$ and 
(by abuse of notation) $\lambda({\mathrm{diag}}(A))$  is the decreasing rearrangement of ${\mathrm{diag}}(A)$. 
 
\begin{theorem}\label{diag majorization}
{\it Let $\phi:\R\rightarrow \R$ be a nonnegative sublinear function. Suppose $A\in \Sn$ is  a positive semidefinite matrix and $b\in \V$. Then, relative to any Jordan frame,
$\phi\big (A\bullet b\big )\underset{w}{\prec} A\bullet \phi(b)$ and 
 $$\lambda \big (\phi(A\bullet b)\big )\underset{w}{\prec} \lambda \big (\mathrm{diag}(A)\big )*\lambda\big (\phi(b)\big ).$$
In particular,
$$\lambda \big (|A\bullet b|\big )\underset{w}{\prec} \lambda \big (\mathrm{diag}(A)\big )*\lambda\big (|b|\big ).$$
}
\end{theorem}

\begin{proof}
We fix a Jordan frame $\{e_1,e_2,\ldots, e_n\}$ relative to which all Schur products are defined. Since $A$ is positive semidefinite,  
the map $P:x\mapsto A\bullet x$ is a positive linear transformation (\cite{gowda-tao-sznajder}, Prop.2.2). The inequality $\phi(A\bullet b)\underset{w}{\prec} A\bullet \phi(b)$ comes from an application of the above lemma. 
Now, letting $A=[a_{ij}]$, we  define $a:=a_{11}e_1+\cdots+a_{nn}e_n$ and note that $a\geq 0$.  
From Proposition \ref{prop4} we have 
$A\bullet b\prec P_{\sqrt{a}}(b) $
for all $b\in \V$. This implies, by the convexity of $\phi$,  $\phi(A\bullet b)\underset{w}{\prec} \phi\big (P_{\sqrt{a}}(b)\big )$.
 Applying the previous theorem, 
$$\lambda (\phi(A\bullet b))\underset{w}{\prec} \lambda(a)*\lambda(\phi(b))= \lambda (\mathrm{diag}(A))*\lambda(\phi(b)).$$

\end{proof}

%%%%%%%%%%%%%%%%%%%%%%%%%%%%%%%%%%%%%%%%%%%%%%%%%%%%%%%%%%%%%%%%%%%%%%%%%%
\section{A weak-majorization inequality}
In this section, we prove the following.

\begin{theorem}\label{weak majorization theorem} 
{\it Let $a, \, b \in \V$. Then
\begin{eqnarray}\label{sb4}
\lambda\big (|a\circ b|\big )\underset{w}{\prec} \lambda(|a|)*\lambda(|b|).
\end{eqnarray}
}
\end{theorem}

As in the case of Theorem \ref{log-majorization theorem}, the motivation comes from matrix theory. Consider  two  matrices $A,B\in \Hn$. 
Then, in the algebra $\Hn$, $A\circ B=\frac{AB+BA}{2}$ and so,
$$\lambda(|A\circ B|)=s(A\circ B)=s \Big (\frac{AB+BA}{2}\Big ),$$
 where we recall that $s(X)$ denotes the vector of singular values of a 
matrix $X$ written in the decreasing order. Invoking the inequality $s(X+Y)\underset{w}{\prec} s(X)+s(Y)$ (\cite{hiai-petz}, Corollary 6.12), we see that 
$$s\Big (\frac{AB+BA}{2}\Big )\underset{w}{\prec}\frac{s(AB)+s(BA)}{2}\underset{w}{\prec} \frac{s(A)*s(B)+s(B)*s(A)}{2}=
\lambda(|A|)*\lambda(|B|),$$
as $s(A)=\lambda(|A|)$, $s(AB)\underset{w}{\prec}s(A)*s(B)$, etc. Thus, $\lambda(|A\circ B|)\underset{w}{\prec} \lambda(|A|)*\lambda(|B|)$. Theorem \ref{weak majorization theorem} is a 
generalization of this to Euclidean Jordan algebras.
Before considering its proof, we present several lemmas. 

In what follows, we let 
$\varepsilon$ (likewise, $\varepsilon^{\prime}$) be an element in $\V$ 
such as $\varepsilon^2=e$. Clearly, such an element is of the form 
$\sum_{i=1}^n \varepsilon_i e_i$, where $\varepsilon_i=1$ for $i=1, 2, \ldots, k$ and $\varepsilon_j=-1$ for $j=k+1, k+2, \ldots, n$ ($1\leq k \leq n$) for some  Jordan frame $\{e_1, e_2,\ldots, e_n \}$.

\begin{lemma}\label{sb1} 
{\it 
Let $\varepsilon\in \V$ with $\varepsilon^2=e$ and $b \in \V$. Then
$\lambda (|b\circ \varepsilon |) \underset{w}{\prec} \lambda(|b|)*\lambda(|\varepsilon|)=\lambda(|b|).$
}
\end{lemma}

\noindent{\bf Proof.} As the result is obvious when $\varepsilon =e$ or $-e$, we  assume the spectral decomposition of $\varepsilon$ in the form  
$\varepsilon=\sum_{i=1}^n \varepsilon_i e_i$, where $1\leq k<n$, 
$\varepsilon_i=1$ for $i=1, 2, \ldots, k$ and $\varepsilon_j=-1$ for $j=k+1, k+2, \ldots, n$.
We let $c=\sum_1^ke_i$ and consider the Peirce decomposition  
$b=u+v+w$, where  $u\in \V(c,1)$, $v\in \V(c,\frac{1}{2})$ and $w\in \V(c,0)$.
A direct calculation leads to $b\circ \varepsilon=u-w$. Now, writing the spectral decompositions of $u$ and $w$ in the (sub)algebras $\V(c,1)$ and $\V(c,0)$
in the form $u=\sum_{1}^{k}u_if_i$ and $w=\sum_{k+1}^{n}w_if_i$ for some Jordan frame $\{f_1,f_2,\ldots, f_n\}$ in $\V$, we see that
$$|u-w|=\abs{\sum_{1}^{k}u_if_i-\sum_{k+1}^{n}w_if_i}=\sum_{1}^{k}|u_i|f_i+\sum_{k+1}^{n}|w_i|f_i=|u+w|.$$ Thus, $\lambda(|u-w|)=\lambda(|u+w|).$ 
By Proposition \ref{prop2}, $u+w\prec b$ and so
$\lambda(u+w)\prec \lambda(b)$. By Item $(b)$ in Proposition \ref{prop1}, $|\lambda(u+w)|\underset{w}{\prec} |\lambda(b)|.$ Since $\lambda(|b|)$ is  the decreasing rearrangement of the vector
$|\lambda(b)|$, we see that 
$\lambda(|b\circ \varepsilon|) =\lambda(|u-w|)=\lambda(|u+w|)\underset{w}{\prec} \lambda(|b|)$.
This completes the proof.
\hfill $\Box$

\gap

We note a simple consequence. 

\begin{corollary}
{\it 
For elements $\varepsilon$ and $\varepsilon^\prime$ with $\varepsilon^2=e=(\varepsilon^\prime)^2$ and any idempotent $c$ in $\V$, we have
$$\lambda (|c\circ \varepsilon|) \underset{w}{\prec} \lambda(c)\quad\mbox{and}\quad 
\lambda (|\varepsilon^{\prime}\circ (c\circ \varepsilon)|) \underset{w}{\prec} \lambda(c).$$
}
\end{corollary}

\begin{lemma}\label{sb3} 
{\it 
For elements $\varepsilon$ and $\varepsilon^\prime$ with $\varepsilon^2=e=(\varepsilon^\prime)^2$ and idempotents $c$ and $c^\prime$ in $\V$, we have
$$\lambda \big (|(\varepsilon^{\prime}\circ c') \circ (c\circ \varepsilon)|\big ) \underset{w}{\prec} \lambda(c)*\lambda(c').$$
}
\end{lemma}

\noindent{\bf Proof.} Let $m=\min\{\mbox{rank}(c),\mbox{rank}(c^\prime)\}$. Note that $\lambda(c)*\lambda(c^\prime)$ is a vector in $\Rn$ with ones in the first $m$ slots and 
zeros elsewhere. From the previous corollary, $\lambda_1(|c\circ \varepsilon)|)\leq 1$ and $\lambda_1(|\varepsilon^{\prime}\circ c'|)\leq 1$;  hence, by Item $(vi)$ in Proposition \ref{prop3}, 
$\lambda_1\big (|(\varepsilon^{\prime}\circ c') \circ (c\circ \varepsilon)|\big ) \leq \lambda_1\big (|\varepsilon^{\prime}\circ c'|\big )\, \lambda_1\big (|c\circ \varepsilon)|\big)\leq 1$. Thus,
$$\lambda_i\big (|(\varepsilon^{\prime}\circ c') \circ (c\circ \varepsilon)|\big ) \leq 1, \quad i=1, 2, \ldots, n.$$
Additionally, for some $\varepsilon^{\prime\prime}$ with $(\varepsilon^{\prime\prime})^2=e$,
\begin{eqnarray*}
\tr \big (|(\varepsilon^{\prime}\circ c') \circ (c\circ \varepsilon)|\big )&=&\big \langle [(\varepsilon^{\prime}\circ c') \circ (c\circ \varepsilon)]\circ \varepsilon^{\prime\prime},  e\big \rangle=\big \langle c\circ \varepsilon, \varepsilon^{\prime\prime} \circ(\varepsilon^{\prime}\circ c') \big \rangle \\
                                                                                                &\leq& \big \langle \lambda(c\circ \varepsilon), \lambda(\varepsilon^{\prime\prime} \circ(\varepsilon^{\prime}\circ c')) \big \rangle 
\\
&\leq& \big \langle \lambda\big (|c\circ \varepsilon|\big ), \lambda\big (|\varepsilon^{\prime\prime} \circ(\varepsilon^{\prime}\circ c')|\big ) \big\rangle\\                                                                                                 &\leq& m,
\end{eqnarray*}
where the first two inequalities come from item $(ii)$ in Proposition \ref{prop3} and the last inequality is due to the previous lemma and Item $(a)$ in Proposition \ref{prop1}. 
Hence,
$$\lambda \big (|(\varepsilon^{\prime}\circ c') \circ (c\circ \varepsilon)|\big ) \underset{w}{\prec} \lambda(c)*\lambda(c').$$
\hfill $\Box$

\begin{lemma}\label{sb5} 
{\it 
For any $\varepsilon$ with $\varepsilon^2=e$, idempotent $c$, and $b$ in $\V$, we have
$$\lambda \big (|b\circ(c\circ  \varepsilon)|\big ) \underset{w}{\prec} \lambda(|b|)*\lambda(c).$$
}
\end{lemma}

\noindent{\bf Proof.} We write the spectral decomposition  $b\circ(c\circ  \varepsilon)=\sum_{i=1}^n \lambda_ie_i$. As the conclusion of the lemma remains the same if $b$ is replaced by $-b$, we may assume that some $\lambda_i$ is nonnegative. Without loss of generality, let 
$\lambda_i\geq 0$, for
$i=1, 2, \ldots, k$ and $\lambda_j < 0$ for $j=k+1, k+2, \ldots, n$, where $k\in \{1,2,\ldots, n\}$. (This includes the possibility that $k=n$, in which case, there is no $\lambda_j<0$.) Define 
$\varepsilon^{\prime}:=\sum_{i=1}^n \varepsilon_i^{\prime} e_i$, where $\varepsilon_i^{\prime}=1$ for $i=1, 2, \ldots, k$ and $\varepsilon_j^{\prime}=-1$ for $j=k+1, k+2, \ldots, n$.
Then, for any idempotent $f$ of rank $l$, that is, $f\in \J^{(l)}(\V)$, we have
\begin{eqnarray*}
\big \langle |b\circ(c\circ \varepsilon)|,  f\big \rangle=\big \langle (b\circ(c\circ \varepsilon))\circ \varepsilon^{\prime},  f\big \rangle &=& \langle b,  (c\circ \varepsilon) \circ (\varepsilon^{\prime}\circ f)\rangle \\
&\leq&  \langle \lambda(b),  \lambda((c\circ \varepsilon) \circ (\varepsilon^{\prime}\circ f)) \rangle \\
        &\leq& \langle \lambda(|b|),  \lambda(|(c\circ \varepsilon) \circ (\varepsilon^{\prime}\circ f)|)\rangle \\
        &\leq& \langle \lambda(|b|), \lambda(c)*\lambda(f) \rangle\\
        &=& \langle \lambda(|b|)*\lambda(c), \lambda(f) \rangle\\
        &=& \sum_{i=1}^{l}(\lambda(|b|)*\lambda(c))_i,
\end{eqnarray*}
where the first two inequalities come from item $(ii)$ in Proposition \ref{prop3}  and the  last inequality is from  the previous Lemma. Now, taking the maximum over $f\in \J^{(l)}(\V)$ and using Item $(i)$ in Proposition \ref{prop3},
we get 
$$\sum_{i=1}^{l}\lambda_i\big (|b\circ(c\circ  \varepsilon)|\big )\leq \sum_{i=1}^{l}(\lambda(|b|)*\lambda(c))_i=\sum_{i=1}^{l}\lambda_i(|b|)\,\lambda_i(c),$$
that is,
$$\lambda (|b\circ(c\circ  \varepsilon)|) \underset{w}{\prec} \lambda(|b|)*\lambda(c).$$
\hfill $\Box$

\gap

\noindent{\bf Proof of Theorem \ref{weak majorization theorem}.} 
We write the spectral decomposition $a \circ b =\sum_{i=1}^n \lambda_i(a\circ b)\, e_i$. As the conclusion of the theorem remains the same if $b$ is replaced by $-b$, we may assume that some $ \lambda_i(a\circ b)$ is nonnegative. Without loss of generality, we assume that for some $k\in \{1,2,\dots,n\}$, 
$\lambda_i (a \circ b) \geq 0$, for
$i=1, 2, \ldots,k $ and $\lambda_j (a \circ b) < 0$ for $j=k+1, k+2, \ldots, n$. ((This includes the possibility that $k=n$.) Define 
$\varepsilon:=\sum_{i=1}^n \varepsilon_i e_i$, where $\varepsilon_i=1$ for $i=1, 2, \ldots, k$ and $\varepsilon_j=-1$ for $j=k+1, k+2, \ldots, n$.
We fix an index $l$, $1\leq l\leq n$, and $c\in \J^{(l)}(\V)$. Then,
\begin{eqnarray*}
\big \langle |a \circ b|,  c\big \rangle              &= &\big \langle (a \circ b) \circ \varepsilon,  c\big \rangle \\
                                                 &= &\big \langle a \circ b,  c \circ \varepsilon \big \rangle \\
                                             &= &\big \langle a,  b\circ (c \circ \varepsilon) \big \rangle \\
                                             &\leq &\big \langle \lambda(|a|),  \lambda(|b\circ (c \circ \varepsilon)|) \big \rangle \\
                                              &\leq  &\big \langle \lambda(|a|),  \lambda(|b|)*\lambda(c) \big \rangle\\
                                              &=& \sum_1^l\lambda_i(|a|)\lambda_i(|b|),
\end{eqnarray*}
where the last inequality is due to  Lemma \ref{sb5} and Item $(a)$ of Proposition \ref{prop1}.
Taking the maximum over all such $c$ and using Item $(i)$ in Proposition \ref{prop3}, we see that 
$$\sum_1^l\lambda_i (|a\circ b|) \leq \sum_1^l\lambda_i(|a|)\lambda_i(|b|).$$
This gives the inequality (\ref{sb4}). 
\hfill $\Box$

\gap

The following example show that the inequalities $\lambda(\abs{a \circ b}) \prec_w \lambda(\abs{a} \circ \abs{b})$ and $\lambda(\abs{a} \circ \abs{b}) \prec_w \lambda(\abs{a \circ b})$ need not hold.

\begin{example}
        In $\mathcal{S}^2$, consider two matrices
        \[ A = \begin{bmatrix} \,8 & 3\, \\ \,3 & 0\, \end{bmatrix} \quad \text{and} \quad B = \begin{bmatrix} \,0 & 3\, \\ \,3 & 8\, \end{bmatrix}. \]
        Then, a direct (or Matlab) calculation shows that $\lambda(\abs{A \circ B}) = (33,\, 15)$ and $\lambda(\abs{A} \circ \abs{B}) = (44.52,\, -3.48)$. Clearly,  the inequalities
$\lambda(\abs{A \circ B}) \prec_w \lambda(\abs{A} \circ \abs{B})$ and $\lambda(\abs{A} \circ \abs{B}) \prec_w \lambda(\abs{A \circ B})$ do not hold.
\end{example}

\gap

\noindent{\bf Remarks.} Theorem \ref{diag majorization} shows that the inequality 
$\lambda(|A\bullet b|)\underset{w}{\prec} \lambda(|{\mathrm{diag}}(A)|)*\lambda(|b|)$  holds for any   $A\in \Sn$ that is positive semidefinite.  
Now, the inequality $\lambda\big (|a\circ b|\big )\underset{w}{\prec} \lambda(|a|)*\lambda(|b|)$ can be viewed as $\lambda(|A\bullet b|)\underset{w}{\prec} \lambda({|\mathrm{diag}}(A)|)*\lambda(|b|)$, where $a$ has the spectral decomposition $a=a_1e_1+a_2e_2+\cdots+a_ne_n$, $A=[\frac{a_i+a_j}{2}]$, and the Schur product $A\bullet b$ is defined relative to the Jordan frame
$\{e_1,e_2,\ldots, e_n\}$. However, this $A$ is not, in general, positive semidefinite. Motivated by this observation/result, we raise the following \\

{\bf Problem}: {\it Characterize $A\in \Sn$ such that $\lambda(|A\bullet b|)\underset{w}{\prec} \lambda(|{\mathrm{diag}}(A)|)*\lambda(|b|)$ holds for all $b$. }\\

A related weaker problem could be: {\it Characterize $A\in \Sn$ such that $\lambda(A\bullet b)\underset{w}{\prec} \lambda(|{\mathrm{diag}}(A)|)*\lambda(b)$ holds for all $b\geq 0$. }

%%%%%%%%%%%%%%%%%%%%%%%%%%%%%%%%%%%%%%%%%%%%%%%%%%%%%%%%%%%%%%%%%%%%%%%%%%%%
\section{A generalized H\"{o}lder type inequality and norms of Lyapunov and quadratic representations}
Recall that for $p\in [1,\infty]$, the spectral $p$-norm on $\V$ is given by $||x||_p:=||\lambda(x)||_p$, where the latter norm is computed in $\Rn$. When numbers $r,s\in [1,\infty]$ are conjugates,  that is, when $\frac{1}{r}+\frac{1}{s}=1$, 
the following H\"{o}lder type inequality holds \cite{gowda-holder}:
$$||x\circ y||_1\leq ||x||_r\,||y||_s\quad (x,y\in \V).$$
It was conjectured in \cite{gowda-sznajder}, Page 9, that a generalized version, namely, 
$$||x\circ y||_p\leq ||x||_r\,||y||_s\quad (x,y\in \V)$$ holds for  $p,r,s\in [1,\infty]$ with $\frac{1}{p}=\frac{1}{r}+\frac{1}{s}$.
In what follows, we  settle this conjecture in the affirmative. 
\\
For  a linear transformation $T:(\V, ||\cdot||_r)\rightarrow (\V, ||\cdot||_s)$, we define the corresponding norm by  
$$||T||_{r\rightarrow s}:=\sup_{0\neq x\in \V}\frac{||T(x)||_s}{||x||_r}.$$
The problem of finding these norms for $L_a$ and $P_a$ has been
addressed in two recent papers \cite{gowda-holder, gowda-sznajder}, where only partial results were given. As a consequence of the following result, 
 we give a complete description of the norms of $L_a$ and $P_a$  relative to two spectral norms.

\begin{theorem}\label{holder type theorem for A}
{\it Given $A\in \Sn$ and  a Jordan frame in $\V$, we consider the Schur product $A\bullet x$ for any $x\in \V$ and define the linear transformation $D_A$ on $\V$ by 
$D_A(x):=A\bullet x$. Suppose the inequality
\begin{equation}\label{majorization assumption}
\lambda \big (|A\bullet x|\big )\underset{w}{\prec} \lambda \big (|\mathrm{diag}(A)|\big )*\lambda\big (|x|\big ) 
\end{equation}
holds for all $x\in \V$. Let $p,r,s\in [1,\infty]$ with $\frac{1}{p}=\frac{1}{r}+\frac{1}{s}$. Then,  
\begin{equation} \label{holder condition}
||A\bullet x||_p\leq ||\mathrm{diag}(A)||_r\,||x||_s\quad (x\in \V).
\end{equation}
Consequently,
$$||D_A||_{r\rightarrow s}=\left \{
\begin{array}{lll}
||\mathrm{diag}(A)||_\infty & if & r\leq s,\\
||\mathrm{diag}(A)||_{\frac{rs}{r-s}} & if &  s <r.
\end{array}
\right .
$$
In particular, the above conclusions hold when $A$ is positive semidefinite.
}
\end{theorem}

\begin{proof} Consider a fixed nonzero $x\in \V$.
To simplify the notation, let $z:=\lambda \big (|A\bullet x|\big )$, $u:=\lambda \big (|\mathrm{diag}(A)|\big )$, and $v:=\lambda\big (|x|\big )$, which are nonnegative vectors (with decreasing entries) in $\Rn$. Then, by our assumption, $z\underset{w}{\prec} u*v$ in $\Rn$. We consider two cases:\\
\noindent{\it Case 1:} $p=\infty$ (so $r=s=\infty$).\\  Then, by comparing the first components in $u,v,z$, we have $z_1\leq u_1\,v_1$, that is, $||z||_\infty\leq ||u||_\infty\,||v||_\infty$, where the norms are computed in $\Rn$. 
This gives, $||A\bullet x||_\infty\leq ||\mathrm{diag}(A)||_\infty\,||x||_\infty.$
\\
\noindent{\it Case 2:} $p<\infty$. \\
Then,
$$||z||_p\leq ||u*v||_p\leq ||u||_r\,||v||_s,$$
where the first inequality comes from Item $(c)$ in Proposition \ref{prop1} by considering the increasing convex function $t\mapsto t^p$ and the second inequality comes from  the classical generalized H\"{o}lder's inequality in $\Rn$. From this, we get $||A\bullet x||_p\leq ||\mathrm{diag}(A)||_r\,||x||_s.$ \\
Thus we have proved (\ref{holder condition}) in both cases.\\

Now for the computation of $||D_A||_{r\rightarrow s}.$ Let $\{e_1,e_2,\ldots, e_n\}$ be the Jordan frame that defines our Schur products. Let $A=[a_{ij}]$.
We consider two cases:\\
\noindent{\it Case $(i)$:} $r\leq s$. \\
Then,
$$\frac{||A\bullet x||_s }{||x||_r}\leq \frac{||A\bullet x||_s }{||x||_s}\leq ||\mathrm{diag}(A)||_\infty,$$
where the first inequality is due to the fact that $r\leq s\Rightarrow ||x||_s\leq ||x||_r$ (because, in $\Rn$, for a fixed vector $v$, the norm $||v||_t$ decreases in $t$) and the second inequality is due to the fact that $||z||_s\leq ||u*v||_s\leq ||u||_\infty\,||v||_s$ in $\Rn$.
Thus, 
$$||D_A||_{r\rightarrow s}\leq ||\mathrm{diag}(A)||_\infty.$$
To see the reverse inequality, we observe that  
$A\bullet e_i=a_{ii}e_i$ for all $1\leq i\leq n$. As $1$ is the only nonzero eigenvalue of $e_i$, $||e_i||_r=1$  and so, $|a_{ii}|=\frac{||A\bullet e_i||_s}{||e_i||_r}$ for all $i$. Thus,
$$||\mathrm{diag}(A)||_\infty \leq \max_{1\leq i\leq n}\frac{||A\bullet e_i||_s}{||e_i||_r}\leq ||D_A||_{r\rightarrow s}.$$
Hence,
$$||D_A||_{r\rightarrow s}=||\mathrm{diag}(A)||_\infty.$$
\noindent{\it Case $(ii)$:} $s<r$ (so $s<\infty$).\\
We define $t\in [1,\infty)$ by $\frac{1}{s}=\frac{1}{t}+\frac{1}{r}$ so that $t=\frac{rs}{r-s}$ (which is taken to be $1$ when $r=\infty$). Then, by an application of (\ref{holder condition}), $||A\bullet x||_s\leq ||\mathrm{diag}(A)||_t\,||x||_r$. This implies 
$$\frac{||A\bullet x||_s }{||x||_r}\leq \frac{||\mathrm{diag}(A)||_t\,||x||_r }{||x||_r}= ||\mathrm{diag}(A)||_t.$$
From this, we get 
$$||D_A||_{r\rightarrow s}\leq ||\mathrm{diag}(A)||_t.$$ As this turns into equality when $\mathrm{diag}(A)=0$, we  prove the reverse inequality by assuming $\mathrm{diag}(A)\neq 0$.
In this setting,  let  
$x=\sum_{i=1}^{n}|a_{ii}|^{\frac{t}{r}}(sgn\,a_{ii})\,e_i$ (using the convention $0^0=1$, if needed, when $r=\infty$), where $sgn\,a_{ii}$ denotes the sign of $a_{ii}$. Then, $A\bullet x=\sum_{i=1}^{n} |a_{ii}|^{\frac{t}{r}+1}e_i$. So,
$||A\bullet x||_s=||\mathrm{diag}(A)||_{t}^{\frac{t}{s}}.$ Also, $||x||_r=||\mathrm{diag}(A)||_{t}^{\frac{t}{r}}.$ Thus, for this $x$,
$$\frac{||A\bullet x||_s}{||x||_r}=\frac{||\mathrm{diag}(A)||_{t}^{\frac{t}{s}}}{||\mathrm{diag}(A)||_{t}^{\frac{t}{r}}}=||\mathrm{diag}(A)||_{t}.$$ 
It follows that $||D_A||_{r\rightarrow s}\geq ||\mathrm{diag}(A)||_t.$ Hence, 
$$||D_A||_{r\rightarrow s}= ||\mathrm{diag}(A)||_t.$$  
Finally, when $A$ is positive semidefinite, condition (\ref{majorization assumption}) holds thanks to Theorem \ref{diag majorization}. Hence (\ref{holder condition}) 
and norm statements hold in this special case as well.
\end{proof}

We now describe the norms of $L_a$ and $P_a$.

\begin{corollary}
{\it Consider  $p,r,s\in [1,\infty]$ with $\frac{1}{p}=\frac{1}{r}+\frac{1}{s}$. Then the following statements hold for all  $a,b\in \V$:
\begin{itemize}
\item [(i)] $$||a\circ b||_p\leq ||a||_r\,||b||_s\quad\mbox{and}\quad  
||L_a||_{r\rightarrow s}=\left \{
\begin{array}{lcl}
 ||a||_\infty & if & r\leq s,\\
 ||a||_{\frac{rs}{r-s}} & if & s <r.
\end{array}
\right .
$$
\item [(ii)] $$||P_a(b)||_p\leq ||a^2||_r\,||b||_s\quad \mbox{and}\quad
||P_a||_{r\rightarrow s}=\left \{
\begin{array}{lcl}
 ||a^2||_\infty & if & r\leq s,\\
 ||a^2||_{\frac{rs}{r-s}} & if & s <r.
\end{array}
\right .
$$
\end{itemize}
}
\end{corollary}

\begin{proof}
We prove Item $(i)$. For a fixed $a\in \V$, we consider its spectral decomposition $a=a_1e_1+a_2e_2+\cdots+a_ne_n$ and define all Schur products relative to $\{e_1,e_2,\ldots, e_n\}$. Then, with $A=[\frac{a_i+a_j}{2}]$, $L_a(x)=A\bullet x$ for all $x\in \V$. Since 
 condition (\ref{majorization assumption}) holds (thanks to Theorem \ref{weak majorization theorem}), we can apply Theorem \ref{holder type theorem for A}. As
 the diagonal entries  of $A$ are $a_1,a_2,\ldots, a_n$, which are the eigenvalues of $a$, we  get both the items in $(i)$. Based on the matrix $A=[a_ia_j]$, a similar argument gives $(ii)$.
\end{proof}

%%%%%%%%%%%%%%%%%%%%%%%%%%%%%%%%%%%%%%%%%%%%%%%%%%%%%%%%%%%%%%%%%%%%%%%%%%%%%%%%%%%

\section*{Acknowledgment}
The second author was financially supported by the National Research Foundation of Korea NRF-2016R1A5A1008055.

%%%%%%%%%%%%%%%%%%%%%%%%%%%%%%%%%%%%%%%%%%%%%%%%%%%%%%%%%%


\begin{thebibliography}{}


\bibitem{baes} M. Baes, {\it Convexity and differentiability properties of spectral functions and spectral mappings on Euclidean Jordan algebras,} Linear Algebra Appl., 422 (2007) 664-700. 

\bibitem{bhatia} R. Bhatia, \emph{Matrix Analysis}, Springer-Verlag, New York, 1997.

\bibitem{cvetkovski}
Z. Cvetkovski, \emph{Inequalities: Theorems, Techniques and Selected Problems}, Springer-Verlag, New York, 2010.

\bibitem{faraut-koranyi} J. Faraut and A. Kor\'{a}nyi, \emph{Analysis on Symmetric Cones}, Oxford University Press, Oxford, 1994.


\bibitem{gowda-doubly stochastic} M.S. Gowda, \emph{
Positive and doubly stochastic maps, and majorization in Euclidean Jordan algebras}, Linear Algebra  Appl., 528 (2017) 40-61. 

\bibitem{gowda-holder} M.S. Gowda, {\it A H\"{o}lder type inequality and an interpolation theorem in Euclidean Jordan algebras}, Jour. Math. Anal. Appl., 474 (2019) 248-263.

\bibitem{gowda-majorization} M.S. Gowda, {\it Some majorization inequalities induced by Schur products in Euclidean Jordan algebras}, Linear Algebra Appl., 600 (2020) 1-21. 

\bibitem{gowda-sznajder} M.S. Gowda and R. Sznajder, \emph{A Riesz-Thorin type interpolation theorem in Euclidean Jordan algebras}, Linear Algebra  Appl., 585 (2020) 178-190.

\bibitem{gowda-sznajder-tao} M.S. Gowda, R. Sznajder, and J. Tao, \emph{P-transformations on Euclidean Jordan algebras,} Linear Algebra Appl., 393 (2004) 203-232.
 
\bibitem{gowda-tao-cauchy} M.S. Gowda and J. Tao, \emph{The Cauchy interlacing theorem in simple Euclidean Jordan algebras and some consequences,} 
Linear and Multilinear Algebra, 59 (2011) 65-86.

\bibitem{gowda-tao-sznajder} M.S. Gowda, J. Tao, and R. Sznajder, \emph{Complementarity properties of Peirce-diagonalizable linear transformations on Euclidean Jordan algebras,} 
Optm. Methods Softw., 27 (2012) 719-733.

\bibitem{hiai} F. Hiai, \emph{Log-majorization and norm inequalities for exponential operators}, Linear Operators, Banach Center Publications, Volume 38,
Institute of Mathematics, Polish Academy of Sciences, Wrszawa 1997. 

\bibitem{hiai-petz} F. Hiai and D. Petz, \emph{Introduction to Matrix Analysis and Applications}, Hindustan Book Agency, New Delhi, 2014.

\bibitem{hirzebruch} U. Hirzebruch, \emph{Der Min-max-satz von E. Fischer f\"{u}r formal-reelle Jordan-algebren}, Math. Ann., 186 (1970), 65-69.
 

\bibitem{jeong et al}  J. Jeong, Y.-M. Jung, and Y. Lim, \emph{Weak majorization, doubly
substochastic maps, and some
related inequalities in Euclidean Jordan algebras}, Linear Algebra Appl.,
597 (2020) 133-154.

\bibitem{li-mathias} C.-K. Li and R. Mathias, \emph{The Lidskii-Mirsky-Wieland theorem-additive and multiplicative versions}, Numer. Math. 81 (1999) 377-413.

\bibitem{lim} Y. Lim, \emph{Geometric means on symmetric cones}, Arch. Math., 75 (2000) 39-45.

\bibitem{marshall-olkin}
A.W. Marshall, I. Olkin, and B.C. Arnold, \emph{Inequalities: Theory of Majorization and its Applications}, Springer-Verlag, New York, 2010.

\bibitem{tao et al} J. Tao, L. Kong, Z. Luo, and N. Xiu, \emph{Some majorization inequalities in Euclidean Jordan algebras}, Linear Algebra Appl.,
161 (2014) 92-122.

\bibitem{wang et al}
G. Wang, J. Tao, and L. Kong, \emph{A note on an inequality involving Jordan product in Euclidean Jordan algebras}, Optim. Lett., 10 (2016) 731-736.

\end{thebibliography}
\end{document}